\documentclass[12pt,reqno,a4paper]{amsart}
\usepackage{amsmath}
\usepackage{amssymb}
\usepackage{bbm}
\makeatletter
\@addtoreset{equation}{section}

 \makeatother

\newcommand{\la}{\langle}
\newcommand{\ra}{\rangle}

\newcommand{\Z}{\mathbb{Z}}
\newcommand{\C}{\mathbb{C}}
\newcommand{\R}{\mathbb{R}}

\newcommand{\Q}{\mathbb{Q}}

\newcommand{\VOA}{\mathrm{VOA}}

\newcommand{\vir}{\mathrm{Vir}}

\newcommand{\om}{\omega}
\newcommand{\be}{\beta}
\newcommand{\al}{\alpha}

\newcommand{\Com}{\mathrm{Com}}

\newcommand{\w}{\omega}

\makeatletter
\@addtoreset{equation}{section}

\makeatother

\theoremstyle{plain}

\newtheorem{thm}{Theorem}[section]
\newtheorem{prop}[thm]{Proposition}
\newtheorem{lem}[thm]{Lemma}

\theoremstyle{definition}
\newtheorem{df}[thm]{Definition}

\newtheorem{nota}[thm]{Notation}
\newtheorem{rem}[thm]{Remark}

\newcommand{\G}{\mathcal{G}}

\newcommand{\tensor}{\otimes}

\title[Majorana representation of $3^2{:}2$ ]{On Majorana representations of the group $3^2{:}2$ of $3C$-pure type and the corresponding vertex operator algebras}
 \author{Hsian-Yang Chen}
  \address[H.Y.  Chen]{ Institute of Mathematics, Academia Sinica, Taipei  10617, Taiwan}
\email{hychen@math.sinica.edu.tw}
 \author{Ching Hung Lam}
 \address[C.H. Lam]{ Institute of Mathematics, Academia Sinica, Taipei  10617, Taiwan
 and National Center for Theoretical Sciences, Taiwan}
\email{chlam@math.sinica.edu.tw}
\thanks{Partially supported by NSC grant
  100-2628-M-001005-MY4 }


\makeatother

\begin{document}

\begin{abstract}
In this article, we study Griess algebras and vertex operator subalgebras  generated by Ising vectors in a moonshine
type VOA such that the subgroup generated by the corresponding Miyamoto involutions has the shape $3^2{:}2$ and any two
Ising vectors generate a 3C subVOA $U_{3C}$. We show that such a Griess algebra is uniquely determined, up to isomorphisms.
The structure of the corresponding vertex operator algebra is also discussed. In addition, we give a construction  of  such
a VOA inside  the lattice VOA $V_{E_8^3}$, which  gives an explicit example for Majorana representations of the group $3^2{:}2$ of $3C$-pure type.
\end{abstract}

\maketitle

\section{Introduction}

A vertex operator algebra (VOA) $V=\oplus_{n=0}^{\infty} V_{n}$ is said to be of \emph{moonshine type} if $\dim(V_{0})=1$ and
$V_{1}=0.$ In this case, the weight 2 subspace $V_{2}$ has a commutative non-associative product defined by $a\cdot
b=a_{1}b$ for $a,b\in V_{2}$ and it has a symmetric invariant bilinear form $\langle\cdot,\cdot\rangle$ given by $\langle
a,b\rangle\mathbbm{1}=a_{3}b$ for $a,b\in V_{2}$ \cite{FLM}. The algebra
$(V_{2},\cdot,\langle\cdot,\cdot\rangle)$ is often called the \emph{Griess algebra} of $V$. An element $e\in V_{2}$  is called an
\emph{Ising vector} if  $e\cdot e=2e$ and the sub-VOA generated by $e$ is isomorphic to the simple Virasoro VOA
$L(\frac{1}{2},0)$ of central charge $\frac{1}{2}$. In \cite{miy}, the basic properties of Ising vectors have been studied. Miyamoto
also gave a simple method to construct involutive automorphisms of a VOA $V$ from Ising vectors. These automorphisms are
often called Miyamoto involutions. When $V$ is the famous Moonshine VOA $V^{\natural}$, Miyamoto \cite{M} showed that
there is a $1-1$ correspondence between the $2A$ involutions of the Monster group and Ising vectors in $V^{\natural}$ (see also
\cite{ho}). This correspondence is very useful for  studying some mysterious phenomena of the Monster group and many problems
about $2A$-involutions in the Monster group may also be translated into questions about Ising vectors. For example, the McKay's
observation on the affine $E_{8}$-diagram has been studied in \cite{LYY} using Miyamoto involutions and certain VOA generated
by $2$ Ising vectors were constructed. Nine  different VOA were constructed in \cite{LYY} and they are denoted by
$U_{1A},U_{2A},U_{2B},U_{3A},U_{3C},U_{4A},U_{4B},U_{5A}$, and $U_{6A}$ because of their connection to the 6-transposition
property of the Monster group (cf. \cite[Introduction]{LYY}), where  $1A,2A,\dots,6A$ are the labels for certain conjugacy classes
of the Monster as denoted in \cite{ATLAS}. In \cite{Sa}, Griess algebras generated by 2 Ising vectors contained in a moonshine type
VOA over $\R$ with a positive definite invariant form are classified. There are also $9$ possible cases, and they correspond exactly to
the Griess algebra $\mathcal{G}U_{nX}$ of the $9$ VOA $U_{nX}$, for $nX\in\{1A,2A,2B,3A,3C,4A,4B,5A,6A\}$. Therefore, there
is again a correspondence between the dihedral subgroups generated by two $2A$-involutions, up to conjugacy and the Griess
subalgebras generated by 2 Ising vectors in $V^{\natural}$, up to isomorphism.

Motivated by the work of Sakuma \cite{Sa}, Ivanov
axiomatized the properties of Ising vectors and introduced the notion of Majorana representations for finite groups in his book \cite{Iva}. Ivanov and his research group also initiated a program on classifying the Majorana representations for various finite groups \cite{Iva2,Iva3,IS,IPSS}.  In particular,  the famous $196884$-dimensional Monster Griess algebra constructed by Griess\,\cite{G} is a Majorana representation of the Monster simple group. In fact, most known examples of Majorana representations are constructed as certain subalgebras of this Monster Griess algebra.

In this article, we study the Griess algebras and vertex operator subalgebras  generated by Ising vectors in a moonshine type VOA
such that  the subgroup generated by the corresponding Miyamoto involutions has the shape $3^2{:}2$ and any two Ising vectors
generate a 3C subVOA $U_{3C}$.   We show that such a Griess algebra is uniquely determined, up to isomorphisms. The structure
of the corresponding vertex operator algebra is also discussed. In particular, we determine the central charge. We also give an
explicit construction of such a VOA in a lattice VOA. Therefore,  we obtain an example for a Majorana representation of the group
$3^2{:}2$ of $3C$-pure type. We shall note that the Monster simple group does not have a subgroup of the shape $3^2{:}2$ such
that all order 3 elements belong to the conjugacy class $3C$. Therefore, the VOA that we constructed cannot be embedded into
the Moonshine VOA.

The structure of the Griess algebra can actually be  determined easily by the structure of the $3C$-algebra and the action of the
group $3^2{:}2$. It is also not very difficult to determine the central charge. To construct the VOA explicitly, we combine the
construction of the so-called dihedral subVOA from \cite{LYY} and the construction of $EE_8$-pairs from \cite{GL}. In fact, it is
quite straightforward to find Ising vectors satisfying our hypotheses. The main difficulty is to show the subVOA generated by these
Ising vectors has zero weight 1 subspace.

The organization of this article is as follows. In Section 2, we recall some basic definitions and notations. We also review the
structure of the so-called $3C$-algebra from \cite{LYY,LYY2}. In Section 3, we study the Griess algebra generated 9 Ising vectors
such that the subgroup generated by the corresponding Miyamoto involutions has the shape $3^2:2$ and  any two Ising vectors
generate a 3C subVOA $U_{3C}$. We show that such a Griess algebra is uniquely determined, up to isomorphisms.  We also study
the subVOA  generated by these Ising vectors. We show that such a VOA  has central charge $4$ and it has a full subVOA
isomorphic to $L(\frac{1}2,0)\otimes L(\frac{21}{22}, 0) \otimes L(\frac{28}{11},0).$  Some highest weight vectors are also
determined. In Section 4, we give an explicit construction of a VOA $W$ in the lattice $V_{E_8^3}$ satisfying our hypotheses. In
Section 5, we show that the VOA $W$ is isomorphic to the commutant subVOA $\Com_{V_{E_8^3}}( L_{\hat{sl}_9(\C)}(3,0))$ using
the theory of parafermion VOA. The decomposition of $W$ as a sum of irreducible modules of the parafermion VOA $K(sl_3(\C),9)$ is
also obtained.

\medskip

\section{Preliminary}
First  we will recall some definitions and review several basic facts.

\begin{df}
\label{A-bilinear-form}
Let $V$ be a VOA.
A bilinear $\langle\!\langle\cdot,\cdot\rangle\!\rangle$
form on $V$ is said to be \emph{invariant (or contragredient,} see
\cite{FHL}) if
\begin{equation}
\langle\!\langle Y(a,z)u,v\rangle\!\rangle=\langle\!\langle u,Y(e^{zL(1)}(-z^{-2})^{L(0)}a,z^{-1})v\rangle\!\rangle\label{eq:invariant}
\end{equation}
for any $a$, $u$, $v$ $\in$ $V$.
\end{df}

\begin{df}
Let $V$ be a VOA over $\C$. A \emph{real form} of $V$ is a subVOA  $V_\R$ (with the same vacuum and Virasoro element) of $V$ over $\R$ such that $V=V_\R\otimes \C$.
A real form $V_\R$ is said to be \emph{a positive definite real form} if the invariant form $\langle\!\langle\cdot,\cdot\rangle\!\rangle$ restricted on $V_\R$ is real valued and positive definite.
\end{df}

\begin{df} \label{vvc}
Let $V$ be a VOA. An element $v\in V_2$ is called a \emph{simple Virasoro vector} of central charge $c$
 if the subVOA $\mathrm{Vir}(v)$ generated by $e$ is isomorphic to the simple Virasoro VOA $L(c,0)$ of central charge $c$.
\end{df}

\begin{df} \label{Ising vector}
A simple Virasoro vector of central charge $1/2$ is an \emph{Ising vector}.
\end{df}

\begin{rem} \label{3irr module} It is well known the VOA $L(\frac{1}{2},0)$ is rational
and it has exactly 3 irreducible modules $L(\frac{1}{2},0)$, $L(\frac{1}{2},\frac{1}{2})$,
and $L(\frac{1}{2},\frac{1}{16})$ (cf. \cite{dmz,miy}).
\end{rem}

\begin{rem}\label{Veh}
Let $V$ be a VOA and let $e\in V$ be an Ising vector.
Then we have the decomposition
\[
V=V_{e}(0)\oplus V_{e}(\frac{1}{2})\oplus V_{e}(\frac{1}{16}),
\]
where $V_{e}(h)$ denotes the sum of all irreducible $\mathrm{Vir}(e)$-submodules of
$V$ isomorphic to $L(\frac{1}{2},h)$ for $h\in\{0,\frac{1}{2},\frac{1}{16}\}$.
\end{rem}

\begin{thm}[\cite{miy}]\label{taue}
The  linear map $\tau_{e}:\, V\rightarrow V$ defined by
\begin{equation}
\tau_{e}:=\begin{cases}
1 & \text{on \ensuremath{V_{e}(0)\oplus V_{e}(\frac{1}{2})}},\\
-1 & \text{on \ensuremath{V_{e}(\frac{1}{16})}},
\end{cases}\label{eq:taue}
\end{equation}
is an automorphism of $V$.

On the fixed point subspace $
V^{\tau_{e}}=V_{e}(0)\oplus V_{e}(1/2)$,  the  linear map $\sigma_{e}:\, V^{\tau_{e}}\rightarrow V^{\tau_{e}}$ defined
by
\begin{equation}
\sigma_{e}:=\begin{cases}
1 & \text{on \ensuremath{V_{e}(0)}},\\
-1 & \text{on \ensuremath{V_{e}(\frac{1}{2})}},
\end{cases}\label{eq:-1}
\end{equation}
is an automorphism of $V^{\tau_{e}}$.
\end{thm}

\subsection{3C-algebra}
Next  we recall the properties of the $3C$-algebra $U_{3C}$ from \cite{LYY2} (see also \cite{Sa}).
The followings can be found in \cite[Section 3.9]{LYY2}.
\begin{lem}\label{eiej}
Let $U=U_{3C}$ be the $3C$-algebra defined as in \cite{LYY2}. Then
\begin{enumerate}

\item  $U_1=0$ and $U$ is generated by its weight 2 subspace $U_2$ as a VOA.

\item $\dim U_2=3$ and it is spanned by 3 Ising vectors.

\item There exist exactly 3 Ising vectors in $U_2$, say, $ e^0,e^1, e^2$. Moreover,
we have
\[
(e^i)_1 (e^j) =\frac{1}{32}( e^i+e^j -e^k), \quad \text{ and } \quad \la e^i, e^j\ra =\frac{1}{2^8}
\]
for $i\neq j$ and $\{i,j,k\}=\{0,1,2\}$.

\item The Virasoro element of $U$ is given by $$\frac{32}{33}(e^0+e^1+e^2).$$

\item Let $a=\frac{32}{33}(e^0+e^1+e^2) -e^0$. Then $a$ is a simple Virasoro vector of central charge $21/22$. Moreover, the subVOA generated by $e^0$ and $a$ is isomorphic to $L(\frac{1}2,0)\tensor L(\frac{21}{22},0)$.
\end{enumerate}
\end{lem}

\section{Griess algebras generated by Ising vectors}

In this section,  we study  Griess algebras  generated by Ising vectors in  a  moonshine type VOA $V$ over $\R$ such that the invariant bilinear form is positive definite.  We assume that  (1) the subgroup generated by the corresponding Miyamoto
involutions has the shape $3^2{:}2$; and (2) any two of Ising vectors generate a $3C$-algebra.  We shall show that such a Griess
algebra is unique, up to isomorphism. We also show that the subVOA generated by these Ising vectors has central charge $4$ and
it has a full subVOA isomorphic to $L(\frac{1}2,0)\otimes L(\frac{21}{22}, 0) \otimes L(\frac{28}{11},0).$

First we note that the group $3^2{:}2$ has exactly 9 involutions and all of them are conjugated to each other. Consider a set of 9
Ising vectors $\{ e^0, e^1, \dots, e^8\} \subset V$ such that $\la e^{i}, e^{j}\ra=\frac{1}{2^8}$ and $\tau_{e^i}\tau_{e^j}$ has order $3$ for any $i\neq j$, that means the Griess subalgebra generated by  $e^{i}$ and $e^{j}$ is isomorphic to $\mathcal{G}U_{3C}$. We also assume the subgroup $G$ generated by their
Miyamoto involutions has the shape $3^2{:}2$.

\begin{nota}\label{eijGW}
Let $H=O_3(G)\cong 3^2$, the largest normal 3-subgroup of $G$.   Let $g$ and $h$ be generators of $H$. Then $g$ and $h$ have order 3 and $\la g,h\ra= H\cong 3^2$.
Let $e=e^0$  and denote
\[
e^{i,j} := g^ih^j e \quad \text{ for all } 0\leq i,j \leq 2.
\]
Then  $\{e^{i,j} \mid 0\leq i,j \leq 2\}= \{ e^0, e^1, \dots, e^8\}$. We also denote the Griess subalgebra and the subVOA generated
by $\{e^{i,j}\mid 0\leq i,j \leq 2\}$  by  $\G$  and $W$, respectively. The following lemma shows the uniqueness of Griess algebra of pure 3C type.
\end{nota}

\begin{lem}
 The Griess subalgebra $\G$  is spanned by $\{e^{i,j}\mid 0\leq i,j \leq 2\}$ and $\dim \G=9$.
\end{lem}

\begin{proof}
By  Lemma \ref{eiej},  we have
\[
e^{i,j}\cdot e^{i',j'} =
\begin{cases}
\frac{1}{32}( e^{i,j} + e^{i',j'}- e^{i'', j''}), & \text{ if } (i,j)\neq (i'j'),\\
2e^{i,j} &\text{ if } (i,j)= (i'j'),
\end{cases}
\]
where $i+i'+i''=j+j'+j''=0\mod 3$.

Therefore, $span\{e^{i,j}\mid 0\leq i,j \leq 2\}$ is closed under the Griess algebra product and $\G=span\{e^{i,j}\mid 0\leq i,j \leq 2\}$.
By our assumption, we have the Gram matrix
\[
\begin{pmatrix}
\la e^{i,j}, e^{i',j'}\ra
\end{pmatrix}
_{0\leq i,j,i'j' \leq 2}
=
\begin{pmatrix}
\frac{1}4 & \frac{1}{2^8} & \dots & \frac{1}{2^8}\\
\frac{1}{2^8} & \frac{1}4 &  \dots & \frac{1}{2^8}\\
\vdots &  \vdots & \ddots & \vdots \\
\frac{1}{2^8} & \frac{1}{2^8} & \dots & \frac{1}4
\end{pmatrix}.
\]
It has rank $9$ and hence $\{e^{i,j}\mid 0\leq i,j \leq 2\}$ is a  linearly independent set and $\dim \G=9$.
\end{proof}

Next, we shall give more details about the VOA generated by $ \{ e_{i, j}\}$.
\begin{lem}\label{w}
Let $$\w:= \frac{8}9 \sum_{0\leq i,j \leq 2} e^{i,j}.$$
Then $\w$ is a Virasoro vector of central charge $4$. Moreover, $\w\cdot e^{i,j} =e^{i,j}\cdot \w =2e^{i,j}$
for any $0\leq i,j \leq 2$.  In other words, ${\w}/2$ is the identity element in $\G$.
\end{lem}

\begin{proof}
By Lemma \ref{eiej}, we have
\[
\begin{split}
\w\cdot \w & =\left(\frac{8}{9}\right)^2\left( 2\sum_{0\leq i,j, \leq 2} e^{i,j} +   \sum_{0\leq i,j, \leq 2}
\left(\sum_{0\leq i',j' \leq 2 \atop {(i,j)\neq (i'j') \atop i+i'+i''=j+j'+j''=0 \mod 3} } \frac{1}{32}( e^{i,j} + e^{i',j'}- e^{i'', j''})\right)  \right)\\
  & =\left(  \frac{8}{9}\right)^2\left( 2\sum_{0\leq i,j, \leq 2} e^{i,j}  +  \sum_{0\leq i,j, \leq 2} \frac{8}{32} e^{i,j} \right)\\
  & = \left(  \frac{8}{9}\right)^2 \cdot \frac{9}4 \sum_{0\leq i,j, \leq 2} e^{i,j} = 2\w
\end{split}
\]
and
\[
\la \w,  \w\ra  = \left(\frac{8}{9}\right)^2 \left( \frac{1}{4} \times 9 + \frac{1}{2^8} \times 72\right) = 2.
\]
Therefore, $\w$ is a Virasoro vector of central charge $4$.
That  $\w\cdot e^{i,j} =e^{i,j}\cdot \w =2e^{i,j}$ for any $0\leq i,j \leq 2$ can be verified easily by Lemma \ref{eiej}, also .
\end{proof}

\begin{lem}\label{aiaj}
Let
\[
\begin{split}
a^1= & a= \frac{32}{33} (e^{0,0}+e^{0,1}+e^{0,2}) -e^{0,0},\\
a^2= &  \frac{32}{33} (e^{0,0}+e^{1,0}+e^{2,0}) -e^{0,0},\\
a^3= & \frac{32}{33} (e^{0,0}+e^{1,1}+e^{2,2}) -e^{0,0},\\
a^4= & \frac{32}{33} (e^{0,0}+e^{1,2}+e^{2,1}) -e^{0,0}.
\end{split}
\]
Then $a^1,a^2, a^3,a^4$ are Virasoro vectors of central charge $21/22$. Moreover,
we have
\[
a^i\cdot a^j = \frac{1}{33}( 2a^i+2a^j- a^k-a^\ell),
\]
for $i\neq j$ and $\{i,j,k,\ell\}=\{1,2,3,4\}$.
\end{lem}

\begin{proof}
We only prove the case for $i=1$ and $j=2$. The other cases can be proved similarly.
By Lemma \ref{eiej}, we have
\[
\begin{split}
(e^{0,0}+e^{0,1}+e^{0,2})\cdot e^{0,0} &= 2e^{0,0}+\frac{1}{32} ( e^{0,0}+e^{0,1}-e^{0,2}+ e^{0,0}+e^{0,2} - e^{0,1}) = \frac{33}{16}e^{0,0}\\
(e^{0,0}+e^{0,1}+e^{0,2})\cdot e^{1,0} &= \frac{1}{32} ( e^{0,0}+e^{1,0}-e^{2,0}+ e^{0,1}+e^{1,0}-e^{2,2}+ e^{0,2}+e^{1,0}-e^{2,1})\\
(e^{0,0}+e^{0,1}+e^{0,2})\cdot e^{2,0} &= \frac{1}{32} ( e^{0,0}+e^{2,0}-e^{1,0}+ e^{0,1}+e^{2,0}-e^{1,2}+ e^{0,2}+e^{2,0}-e^{1,1}).
\end{split}
\]
Thus,
\[
\begin{split}
&\qquad a^1\cdot a^2\\
& =\left[ \frac{32}{33} (e^{0,0}+e^{0,1}+e^{0,2}) -e^{0,0}\right] \cdot
\left[ \frac{32}{33} (e^{0,0}+e^{1,0}+e^{2,0}) -e^{0,0}\right]\\
 &= \left(\frac{32}{33}\right)^2 \left[ \frac{33}{16} e^{0,0} +\frac{1}{32} \left (
 2(e^{0,0}+e^{0,1}+e^{0,2} + e^{1,0}+e^{2,0}) - (e^{1,1}+ e^{1,2}+e^{2,1}+e^{2,2}) \right)\right]\\
 & \quad  - 4e^{0,0}+ 2e^{0,0} \\
 &= \frac{32}{33^2} \left[ 2(e^{0,1}+e^{0,2} + e^{1,0}+e^{2,0}) - (e^{1,1}+ e^{1,2}+e^{2,1}+e^{2,2})\right] - \frac{2}{33^2}e^{0,0}\\
 &= \frac{1}{33}( 2a^1+2a^2- a^3-a^4)
\end{split}
\]
as desired.
\end{proof}

\begin{lem}
Let
\[
b^1= \frac{8}9 \sum_{0\leq i,j \leq 2} e^{i,j} -\frac{32}{33} (e^{0,0}+e^{0,1}+e^{0,2}).
\]
Then $b^1$ is a Virasoro vector of central charge $28/11$. Moreover, $e^{0,0}$, $a^1$ and $b^1$ are mutually orthogonal and $\w= e^{0,0}+ a^1 +b^1$. Therefore, $W$ has a full subVOA isomorphic to the tensor product of Virasoro VOA $$L(\frac{1}2,0)\otimes L(\frac{21}{22}, 0) \otimes L(\frac{28}{11},0).$$
\end{lem}

\begin{proof}
It follows from (4), (5) of Lemma \ref{eiej}  and Lemma \ref{w}.
\end{proof}

\begin{lem}
For any $i,j\in \{2,3,4\}, i\neq j$, the element $a^i-a^j$ is a highest weight vector
of weight $(0,1/11, 21/11)$ with respect to $\vir(e^{0,0})\otimes \vir(a^1)\otimes \vir(b^1)$.
\end{lem}

 \begin{proof}
By Lemma \ref{aiaj}, we have
\[
a^1_1(a^i-a^j) =\frac{1}{33}[ (2a^1+2a^i  -a^j -a^k) - (2a^1+2a^j  -a^i -a^k)] =\frac{1}{11}(a^i-a^j),
\]
where $\{i,j,k\}=\{2,3,4\}$.
Since $e^{0,0}_1 a^i=0$ and  $\om_1 a^i=2 a^i$ for all $i$, $a^i-a^j$ is a highest weight vector
of weight $(0,1/11, 21/11)$ with respect to $\vir(e^{0,0})\otimes \vir(a^1)\otimes \vir(b^1)$.
\end{proof}

\begin{lem}\label{1over16}
The vector $e^{0,1}-e^{0,2}$  is a highest weight vector
of weight $(1/16,31/16,0)$ with respect to $\vir(e^{0,0})\otimes \vir(a^1)\otimes \vir(b^1)$.
\end{lem}

 \begin{proof}
By direct calculations, we have
\[
e^{0,0}_1(e^{0,1}-e^{0,2}) =\frac{1}{32}[( e^{0,0} +e^{0,1}-e^{0,2}) - ( e^{0,0} +e^{0,2}-e^{0,1})]= \frac{1}{16} (e^{0,1}-e^{0,2})
\]
and
\[
\frac{32}{33}( e^{0,0} +e^{0,1}+e^{0,2})_1 (e^{0,1}-e^{0,2}) = 2 (e^{0,1}-e^{0,2}).
\]
Since $a^1= \frac{32}{33}( e^{0,0} +e^{0,1}+e^{0,2})-e^{0,0}$ and  $b^1= \om - \frac{32}{33}( e^{0,0} +e^{0,1}+e^{0,2})$, we have $$a^1_1(e^{0,1}-e^{0,2})= \frac{31}{16}(e^{0,1}-e^{0,2})\quad \text{  and }\quad  b^1_1 (e^{0,1}-e^{0,2})=0$$
as desired.
\end{proof}

\begin{lem}
With respect to $\vir(e^{0,0})\otimes \vir(a^1)\otimes \vir(b^1)$,
$(e^{1,0}+e^{1,1}+e^{1,2})- (e^{2,0}+e^{2,1}+e^{2,2})$ is a highest weight vector of weight
$(1/16, 21/176, 20/11)$. On the other hand, $(e^{1,1} -e^{2,2})- (e^{1,2}-e^{2,1})$ and
$(e^{1,0} -e^{2,0})- (e^{1,1}-e^{2,2})$ are highest weight vectors of weight
$(1/16, 5/176, 21/11)$.
\end{lem}

 \begin{proof}
By the same calculations as in Lemma \ref{1over16},
$(e^{1,0}- e^{2,0})$, $(e^{1,1} -e^{2,2})$ and $(e^{1,2}-e^{2,1})$ are $1/16$-eigenvectors of $e^{0,0}_1$.

By Lemma \ref{eiej}, we also have
\[
\frac{32}{33}(e^{0,0}+e^{0,1}+e^{0,2})_1( e^{1,1}- e^{2,2}) =
\frac{1}{33}(4( e^{1,1}- e^{2,2})+(e^{1,0}-e^{2,0})+(e^{1,2}-e^{2,1})).
\]

Let $v=(e^{1,0}+e^{1,1}+e^{1,2})- (e^{2,0}+e^{2,1}+e^{2,2})$. Then
\[
\frac{32}{33}(e^{0,0}+e^{0,1}+e^{0,2})_1 v = \frac{1}{33} (4+1+1)v=\frac{2}{11}v.
\]
Thus, $a_1^1 v= (\frac{2}{11} -\frac{1}{16})v =\frac{21}{176} v$ and $b^1_1v = (2- \frac{2}{11}) v=\frac{20}{11} v$.

Moreover,
\[
\begin{split}
 &\frac{32}{33}(e^{0,0}+e^{0,1}+e^{0,2})_1( (e^{1,1}- e^{2,2}) - (e^{1,2}-e^{2,1})) \\
=&\frac{1}{33} (4-1) ( (e^{1,1}- e^{2,2}) - (e^{1,2}-e^{2,1}))\\
 = &\frac{1}{11} ( (e^{1,1}- e^{2,2}) - (e^{1,2}-e^{2,1})).
\end{split}
\]
Thus, we have $$a_1^1 ( (e^{1,1}- e^{2,2}) - (e^{1,2}-e^{2,1}))= \frac{5}{176} ( (e^{1,1}- e^{2,2}) - (e^{1,2}-e^{2,1}))$$ and $$b^1_1 ( (e^{1,1}- e^{2,2}) - (e^{1,2}-e^{2,1}))v = \frac{21}{11} ( (e^{1,1}- e^{2,2}) - (e^{1,2}-e^{2,1})).$$

The remaining cases can be proved similarly.
\end{proof}

\section{Lattice VOA $V_{E_8\perp E_8\perp E_8}$}

In this section, we shall construct explicitly a VOA $W$ satisfying the hypothesis in Section 3 inside the lattice VOA $V_{E_8\perp
E_8\perp E_8}$.

Our notation for the lattice vertex operator algebra
\begin{equation}\label{VL}
V_L = M(1) \otimes \C\{L\}
\end{equation}
associated with a positive definite even lattice $L$ is standard \cite{FLM}. In particular, ${\mathfrak h}=\C\otimes_{\Z} L$ is
an abelian Lie algebra and we extend the bilinear form to ${\mathfrak h}$ by $\C$-linearity. Also, $\hat {\mathfrak h}={\mathfrak
h}\otimes \C[t,t^{-1}]\oplus \C k$ is the corresponding affine algebra and $\C k$ is the 1-dimensional center of
$\hat{\mathfrak{h}}$. The subspace $M(1)=\C[\al_i(n)|1\leq i\leq d, n<0]$ for a basis $\{\al_1, \dots,\al_d\}$ of $\mathfrak{h}$,
where $\al(n)=\al\otimes t^n,$ is the unique irreducible $\hat{\mathfrak h}$-module such that $\alpha(n)\cdot 1=0$ for all
$\alpha\in {\mathfrak h}$ and $n$ nonnegative, and $k=1.$ Also, $\C\{L\}=span \{e^{\beta}\mid \beta\in L\}$ is the twisted group
algebra of the additive group $L$ such that $e^\be e^\al=(-1)^{\la \al, \be\ra} e^\al e^\be$ for any $\al, \be\in L$. The vacuum
vector $\mathbbm{1}$  of $V_L$ is $1\otimes e^0$ and the Virasoro element $\omega_L$ is
$\frac{1}{2}\sum_{i=1}^d\beta_i(-1)^2\cdot \mathbbm{1}$ where $\{\beta_1,..., \beta_d\}$ is an orthonormal basis of ${\mathfrak
h}.$ For the explicit definition of the corresponding vertex operators, we shall refer to \cite{FLM} for details.

\medskip

\begin{df}\label{tensor}
Let $A$ and $B$ be integral lattices with the inner products $\la \  , \ \ra_A$  and $\la \  ,  \  \ra_B$, respectively. {\it The tensor
product of the lattices} $A$ and $B$ is defined to be the integral lattice which is isomorphic to $A\otimes_\Z B$ as a $\Z$-module
and has the inner product given by
\[
\la \al\otimes \be, \al' \otimes \be'\ra  = \la\al,\al'\ra_A \cdot \la \be,\be'\ra_B, \quad \text{ for any } \al,\al'\in A,\text{ and } \be,\be'\in B.
\]
We simply denote the tensor product of the lattices $A$ and $B$ by $A\otimes B$.
\end{df}

Now let $L=E_8\perp E_8\perp E_8$ be the orthogonal sum of 3 copies of the root lattice of type $E_8$. Set
\begin{equation}\label{MN}
\begin{split}
M& =\{ (\al, -\al,0) \mid \al\in E_8\} <L,\quad \text{  and }\\
N&=\{ (0, \al, -\al) \mid \al\in E_8\} <L.
\end{split}
\end{equation}
Then $M\cong N \cong \sqrt{2}E_8$ and $M+N\cong A_2\otimes E_8$ (see \cite{GL}). Note that there is a third
$\sqrt{2}E_8$-sublattice
\[
\tilde{N} =\{ (\al,0, -\al)\mid \al \in E_8\} < M+N.
\]

We shall fix a (bilinear) 2-cocyle  $\varepsilon_0: E_8\times E_8 \to \Z_2$ such that
\begin{equation}\label{epsilon0}
\begin{split}
\varepsilon_0(\al, \al)  \equiv \frac{1}{2}\la \al, \al\ra  \mod 2,\\
\text{and}\quad  \varepsilon_0(\al, \be) -  \varepsilon_0(\be, \al) \equiv \la \al, \be\ra \mod 2,\\
\end{split}
\end{equation}
for all $\al, \be \in E_8$. Note that such a $2$-cocycle exists (cf. \cite[(6.1.27)-(6.1.29)]{FLM}). Moreover, $e^\al e^{-\al} = - e^0$ for any $\al \in E_8$ such that $ \la \al, \al \ra =2$.

 We shall extend $\varepsilon_0$ to $L$ by
defining
\[
 \varepsilon_0\Big( (\al,\al', \al''), ( \be, \be', \be'')\Big)= \varepsilon_0(\al, \be )+ \varepsilon_0(\al', \be') +\varepsilon_0(\al'', \be'').
 \]
 It is easy to check by direct calculations that $\varepsilon_0$ is trivial on $M$, $N$ or $\tilde{N}$.

\medskip

 Consider the lattice VOA
 \[
 V_L\cong V_{E_8}\otimes V_{E_8}\otimes V_{E_8}.
 \]

\medskip

\begin{df}\label{rho}
Let ${\bf a}  $ be an element of $E_8$ such that
 \[
K:= \{\be\in E_8\mid \la \be, {\bf a} \ra \in 3\Z\} \cong A_8.
 \]
Set $\tilde{{\bf a}}=({\bf a}, -{\bf a}, 0)$ and define an automorphism $\rho$ of $V_L$ by
\[
\rho = \exp\left( \frac{2\pi i}3 \tilde{{\bf a}}(0)\right).
\]
Then $\rho$ has order $3$ and the fixed point subspace $V_M^\rho \cong V_{\sqrt{2}A_8}$.
\end{df}

\begin{nota}\label{efg}
Let $M, N$ and $\rho$ be defined as above. Set
\[
\begin{split}
e& := e_M =\frac{1}{16}\omega_M + \frac{1}{32} \sum_{\al\in M(4)} e^\al,\\
f& := e_N =\frac{1}{16}\omega_N + \frac{1}{32} \sum_{\al\in N(4)} e^\al,\\
e_{ \tilde{N}} & :=\frac{1}{16}\omega_{ \tilde{N}} + \frac{1}{32} \sum_{\al\in { \tilde{N}} (4)} e^\al, \quad \text{ and }\\
e'&:= \rho(e).
\end{split}
\]
It is shown in \cite{DLMN} that $e, f$ and $e_{\tilde{N}}$ are Ising vectors and hence $e'= \rho(e)$ is also an Ising vector (see also \cite{LYY,LYY2}).
\end{nota}

The following lemma can be proved by direct calculations (see \cite{GL,LYY,LYY2}).

\begin{lem}
We have  $\la e, f\ra =\la e, e'\ra= \la f,e'\ra = 1/{2^8}$. Moreover, the subVOA
$\VOA(e,f)$,  $\VOA(e,g)$, $\VOA(f,g)$ generated by $\{e,f\}$, $\{e,e'\}$ and $\{f,e'\}$, respectively, are isomorphic to the $3C$-algebra $U_{3C}$. We also have $ e_M \cdot e_N = 1/{32} ( e_M + e_N - e_{ \tilde{N} })$, and hence $ \tau_e (f) = e_{\tilde{N}}$.
\end{lem}

\begin{nota}\label{Nota:W}
Let $W:=\VOA(e,f,e')$ be the subVOA generated by $e, f$ and $e'$. We also denote $h=\tau_e\tau_f$ and $g=\tau_e\tau_{e'}$. Then $g$ and $h$ both have order $3$.  Note also that $e,f,e'\in V_{M+N}$ and thus $W< V_{M+N} \cong V_{A_2\otimes E_8}$.
\end{nota}

\begin{lem}\label{ghcom}
As automorphisms of $W$, $g$ commutes with $h$.
\end{lem}

\begin{proof}
Recall that $g=\tau_e\tau_{e'} = \rho$ on $V_{L}$ (see \cite{LYY}). Moreover,
$h(e) =f= e_N$ and $h^2(e)= e_{\tilde{N}}$.

By a direct calculation, we have
\[
hg(e)= hgh^{-1} h(e) = \rho^h (e_N),
\]
where $\rho^h = h\rho h^{-1}= \exp\left(\frac{2 \pi i}3 (0 , {\bf a}, -{\bf a})(0)\right)$.

Since $\la (0, \be, -\be), (0 , {\bf a}, -{\bf a})\ra = 2\la \be, {\bf a}\ra$ and
$\la (0, \be, -\be), ({\bf a}, -{\bf a},0)\ra = -\la \be, {\bf a}\ra$, we have
\[
gh(e) =\rho(e_N)=\rho^h(e_N) =hg(e).
\]
Similarly, we have
\[
hg(e') = hg^{2} (e) = (\rho^h)^{2} (e_N), \quad
gh(e') =ghg(e) = g (\rho^h(e_N))=  (\rho^h)^{2} (e_N)
\]
and
\[
hg(f)=hgh(e)=(hgh^{2})h^2(e)= \rho^h (e_{\tilde{N}}) , \quad  gh(f)=g(e_{\tilde{N}})=\rho(e_{\tilde{N}}).
\]
Hence $gh=hg$ on $W$.
\end{proof}

\begin{nota}
For any $0 \leq i,j\leq 2$, denote
\[
e^{i,j}= g^i h^j (e).
\]
In particular, we have
\begin{align*}
&e^{0,0}= e_M, && e^{0,1}=e_N, && e^{0,2}= e_{\tilde{N}},\\
&e^{1,0}= \rho e_M, && e^{1,1}= \rho e_N, && e^{1,2}= \rho e_{\tilde{N}},\\
&e^{2,0}= \rho^2 e_M, &&  e^{2,1}=\rho^2e_N, && e^{2,2}= \rho^2 e_{\tilde{N}}.
\end{align*}
\end{nota}

 \begin{rem}
 By the same methods as in \cite{GL,LYY}, it is also quite straightforward to verify  that $\la e^{i,j}, e^{i'j'}\ra=\frac{1}{2^8}$ whenever  $(i,j)\neq (i',j')$.
 \end{rem}

\begin{lem}
Let $G$ be the subgroup of $Aut(W)$ generated by $\tau_e$, $\tau_f$ and $\tau_{e'}$. Then $G$ has the shape $3^2:2$.
\end{lem}

\begin{proof}
By Lemma \ref{ghcom}, we know the group $ \langle g, h \rangle$ generated by $ g$ and $h$ is isomorphic to $ 3^2$. We  claim that $ G \cong \langle g, h \rangle : \langle \tau_e \rangle$. For this we need to prove

 1) $ \langle g, h \rangle $ is normal in $ G$;\quad  2) $ G = \langle g, h \rangle \langle \tau_e \rangle$ and\quad  3) $ \langle g, h \rangle \cap \langle  \tau_e \rangle =0$. \\
It is easy to prove 2) and 3). For 1), by Lemma \ref{ghcom} we have $ gh = hg$ and hence $ \tau_f \tau_e \tau_{ e'} = \tau_{ e'} \tau_e \tau_f  $. Thus $ \tau_f  h \tau_f =  \tau_f   \tau_e \tau_{ e'}  \tau_f  = \tau_{ e'} \tau_e \tau_f^2 =  \tau_{ e'} \tau_e = h^2 \in \langle g, h \rangle$. Similar computation gives $ \langle g, h \rangle $ is normal in $ G$.
\end{proof}

\section{Affine VOA $L_{\widehat{sl}_9(\C)}(3,0)$} \label{sec:5}

Recall from Definition \ref{rho} that the sublattice
\[
K=  \{\be\in E_8\mid \la \be, {\bf a}\ra \in 3\Z\} \cong A_8.
\]
Thus, we have an embedding
\[
V_{K\perp K\perp K}\cong V_K\otimes V_K\otimes V_K \hookrightarrow  V_L.
\]
It is also well known\,\cite{FLM} that $V_K\cong V_{A_8}$ is an irreducible level 1 representation of the affine Lie algebra $\widehat{sl}_9(\C)$. Moreover, the weight 1 subspace $(V_K)_1$ is a simple Lie algebra isomorphic to $sl_9(\C)$.

Let $\eta_i:K\to K\perp K\perp K$, $i=1,2,3$, be the embedding of $K$ into the $i$-th direct summand of $K\perp K\perp K$, i.e.,
\[
\eta_1(\al)=(\al,0,0),\quad  \eta_2(\al)=(0, \al,0), \quad  \eta_3(\al)=(0, 0,\al),
\]
for any $\al\in K$.

\begin{nota}\label{hef}
For any $\al\in K(2):=\{ \al\in K\mid \la \al, \al\ra =2\}$, set
\[
\begin{split}
H_\al &= (\al, \al, \al)(-1) \cdot \mathbbm{1}, \\
E_\al &= e^{\eta_1(\al)} + e^{\eta_2(\al)} +e^{\eta_3(\al)}.
\end{split}
\]
Then $\{H_\al, E_\al\mid \al\in K(2)\}$ generates a subVOA isomorphic to
the affine VOA $L_{\hat{sl}_9(\C)}(3,0)$ in $V_L$ (see \cite[Proposition 13.1]{DL} and \cite{FZ}). Moreover, the Virasoro element
of $L_{\hat{sl}_9(\C)}(3,0)$ is given by
\[
\Omega= \frac{1}{2(3+9)} \left[ \sum_{k=1}^8 (h^k, h^k,h^k)(-1)^2\cdot \mathbbm{1} +
\sum_{\al\in K(2)} (E_\al)_{-1}(-E_{-\al})\right],
\]
where $\{h^1, \dots, h^8\}$ is an orthonormal basis of $K\otimes \C= E_8\otimes \C$. Note that the dual vector of $ E_ \alpha$ is $-E_{ - \alpha}$.

We also denote $E=\{(\al, \al, \al) \mid \al\in E_8\}< L$. Note that
\[
E=Ann_{L}(M+N) :=\{ \be\in L\mid \la \be, \be' \ra =0 \text{ for all } \be'\in M+N\}.
\]
\end{nota}

\begin{lem}
Denote the Virasoro element of a lattice VOA $V_S$ by $\omega_S$. Then we
have
\[
\begin{split}
\Omega&=\omega_E +\frac{3}4 \omega_{M+N} -\frac{1}{12} \sum_{\al\in K(2)\atop 1\leq i,j\leq 3, i\neq j} e^{\eta_i(\al)-\eta_j(\al)},\\
& = \omega_L -\frac{8}9 \sum_{ 0\leq i,j\leq2} e^{i,j}.
\end{split}
\]
\end{lem}

\begin{proof}
Let $\{h^1, \dots, h^8\}$ be an orthonormal basis of $A_8\otimes \C= E_8\otimes \C$. Then
\[
\begin{split}
\Omega= & \ \frac{1}{2(3+9)} \left[ \sum_{k=1}^8  (h^k, h^k,h^k)(-1)^2\cdot \mathbbm{1}\right.\\
&\left. \ - \sum_{\al\in K(2)} (e^{\eta_1(\al)} +e^{\eta_2(\al)} +e^{\eta_3(\al)})_{-1} ( e^{-\eta_1(\al)}+e^{-\eta_2(\al)}+e^{-\eta_3(\al)})\right ],\\
=& \frac{1}{24} \left[ 6\omega_E + \sum_{\al\in K(2)} \sum_{i=1}^3 \frac{1}2 ( \eta_i(\al)(-2)\cdot \mathbbm{1}
+ \eta_i(\al)(-1)^2\cdot \mathbbm{1}) \right. \\
& \left. \  - 2\sum_{\al\in K(2) \atop 1\leq i,j\leq 3, i\neq j} e^{\eta_i(\al)-\eta_j(\al)}\right],\\
=& \frac{1}4\omega_E + \frac{18}{24} \omega_L -\frac{1}{12} \sum_{\al\in K(2)\atop 1\leq i,j\leq 3, i\neq j} e^{\eta_i(\al)-\eta_j(\al)}.
\end{split}
\]
Since $\omega_L=\omega_{M+N}+\omega_E$, we have
\[
\Omega =\omega_E +\frac{3}4 \omega_{M+N} -\frac{1}{12} \sum_{\al\in K(2)\atop 1\leq i,j\leq 3, i\neq j} e^{\eta_i(\al)-\eta_j(\al)}.
\]

Now let $ \Delta^i := \{ \beta \in E_8(2) \mid  \langle {\bf{a}}, \beta  \rangle = i \mod 3 \Z\}$ for $ i = 0, 1$ and $ 2$. Note that $ \Delta^0 = K(2)$. Then we have
\begin{align*}
 e^{0,0} = e_M = \frac{1}{ 16} \omega_M + \frac{1}{32} \sum_{ i =0}^2 \sum_{ \alpha \in \Delta^i } e^{( \alpha, - \alpha, 0)}, \\
  e^{1,0} =\rho e_M = \frac{1}{ 16} \omega_M + \frac{1}{32}\sum_{ i =0}^2 \sum_{ \alpha \in \Delta^i} \xi^{2i} e^{( \alpha, - \alpha, 0)}, \\
  e^{2,0} = \rho^2 e_M = \frac{1}{ 16} \omega_M + \frac{1}{32}\sum_{ i =0}^2 \sum_{ \alpha \in \Delta^i } \xi^{ i} e^{( \alpha, - \alpha, 0)}.
\end{align*}
Hence
\[
\sum_{ i=0}^2 e^{i, 0} = ( 1 + \rho + \rho^2) e_M= \frac{3}{ 16} \omega_M +  \frac{3}{32} \sum_{ \alpha \in K(2) } e^{( \alpha, - \alpha, 0)} .
\]
Similar computation gives
\begin{align*}
\sum_{ 0 \le i, j \le 2} e^{i, j } =
\frac{3}{ 16}  ( \omega_M + \omega_N + \omega_ { \tilde{N}}) +
\frac{3}{32}\sum_{\al\in K(2)\atop 1\leq i,j\leq 3, i\neq j} e^{\eta_i(\al)-\eta_j(\al)}.
\end{align*}

Recall that $M+N \cong A_2\otimes E_8$. It contains a full rank sublattice isometric to $(\sqrt{2}A_2)^8$ and hence $\om _{M+N}$ is the sum of the conformal elements of each tensor copy of $V_{\sqrt{2}A_2}^{\otimes 8}$. We also note that the conformal element of the lattice VOA $V_{\sqrt{2}A_2}$ is given by
\[
\begin{split}
\om_{\sqrt{2}A_2} &= \frac{1}6 (\al_1(-1)^2+\al_2(-1)^2+\al_3(-1)^2)\cdot \mathbbm{1}\\
&= \frac{2}3\left( \frac{1}{2}(\frac{\al_1(-1)}{\sqrt{2}})^2+\frac{1}{2}(\frac{\al_2(-1)}{\sqrt{2}})^2
+\frac{1}{2}(\frac{\al_3(-1)}{\sqrt{2}})^2 \right)\cdot \mathbbm{1},
\end{split}
\]
where $\al_1, \al_2, \al_3$ are positive roots of a root lattice type $A_2$ \cite{DLMN}.

Thus $ \omega_{ M +N } = \frac{2}{3} (  \omega_M + \omega_N   + \omega_{ \tilde{N}} )$  and we get
\[
\Omega  = \omega_L -\frac{8}9 \sum_{ 0\leq i,j\leq2} e^{i,j},
\]
as desired.
\end{proof}

\begin{lem}\label{coma8}
For any $0 \leq i, j\leq 2$, we have $e^{i, j}\in \mathrm{Com}_{V_L}\left( L_{\widehat{sl}_9(\C)}(3,0)\right)$. Hence $W\subset \mathrm{Com}_{V_L}\left( L_{\widehat{sl}_9(\C)}(3,0)\right)$.
\end{lem}

\begin{proof}
Since $e^{i,j}\in V_{M+N}$ and $E=\{(\al,\al, \al)\mid \al\in E_8\}$ is orthogonal to $M+N$,
it is clear that $(H_\al)_n e^{i,j} =0 $ for all $n\geq 0$.  It is also clear that $(E_\al)_n e^{i,j}=0 $ for any root $\al\in K$ and $n\geq 2$.

Recall from \cite{FLM} that
\[
Y(e^\al, z) = \exp\left(\sum_{n\in \Z^+} \frac{\al(-n)}{n} z^{n}\right)
\exp\left(\sum_{n\in \Z^+} \frac{\al(n)}{-n} z^{-n}\right) e^\al z^\al.
\]
Now let $\sigma=(123)$ be a 3-cycle. Then by direct calculation,  we have
\[
\begin{split}
&\ \ (E_\al)_1 e^{i,j} =(E_\al)_1 (\rho^ih^j e_M) \\
& =(e^{\eta_1(\al)} + e^{\eta_2(\al)} +e^{\eta_3(\al)})_1 \left( \frac{1}{16} \omega_{h^{j}(M)} +\frac{1}{32} \sum_{\al\in \Delta^+(E_8)} \rho^i (e^{(\eta_{\sigma^{j}(1) } -\eta_{\sigma^{j}(2) })(\al)}  + e^{ -(\eta_{\sigma^{j}(1)} -\eta_{\sigma^{j}(2)})(\al)} ) \right)\\
&= \frac{1}{16} \left(\la \al, \al\ra^2 \frac{1}8(e^{\eta_{\sigma^{j}(1)}(\al)}+e^{\eta_{\sigma^{j}(2)}(\al)})\right) +
\frac{1}{32}\varepsilon(\al, -\al)  (e^{\eta_{\sigma^{j}(1)}(\al)}+e^{\eta_{\sigma^{j}(2)}(\al)}) =0,
\end{split}
\]
and
\[
\begin{split}
& \ \ (E_\al)_0 e^{i,j}\\
 & =(e^{\eta_1(\al)} + e^{\eta_2(\al)} +e^{\eta_3(\al)})_0 \left( \frac{1}{16} \omega_{h^{j}(M)} +\frac{1}{32} \sum_{\al\in \Delta^+(E_8)} \rho^i (e^{(\eta_{\sigma^{j}(1) } -\eta_{\sigma^{j}(2) })(\al)}  + e^{ -(\eta_{\sigma^{j}(1)} -\eta_{\sigma^{j}(2)})(\al)} ) \right)\\
&= \frac{1}{16} \left(\la \al, \al\ra^2 \frac{1}8(\eta_{\sigma^{j}(1) }(\al)(-1)e^{\eta_{\sigma^{j}(1) }(\al)} + \eta_{\sigma^{j}(2)}(\al) (-1)e^{\eta_{\sigma^{j}(2)} (\al)})\right. \\
&\left. \quad - 2\la \al, \al\ra \frac{1}{8}\left(\eta_{\sigma^{j}(1) } -\eta_{\sigma^{j}(2) })(\al)(-1) e^{\eta_{\sigma^{j}(1) }(\al)}
- (\eta_{\sigma^{j}(1) } -\eta_{\sigma^{j}(2) })(\al)(-1) e^{\eta_{\sigma^{j}(2) }(\al)}\right)\right)\\
& \quad  +
\frac{1}{32}\varepsilon(\al, -\al)  (\eta_{\sigma^{j}(2) }(\al)(-1) e^{\eta_{\sigma^{j}(1) }(\al)} +\eta_{\sigma^{j}(1)}(\al)(-1) e^{\eta_{\sigma^{j}(2) }(\al)})\\
&=0
\end{split}
\]
for any root $\al\in K$. Therefore, $(E_\al)_n  e^{i,j} =0$ for all $n\geq 0$.  Since $L_{\widehat{sl}_9(\C)}(3,0)$ is generated by $E_\al$ and $H_\al$, we have the desired conclusion.
\end{proof}

\begin{thm}
Let $W$ be the subVOA generated by $e,f$ and $e'$ in $V_L$. Then $W_1=0$.
\end{thm}

\begin{proof}
Since $h(-1)\cdot \mathbbm{1} \in L_{\hat{sl}_9(\C)}(3,0)$ for all $h\in E$, the commutant subVOA
\[
 \mathrm{Com}_{V_L}\left( L_{\widehat{sl}_9(\C)}(3,0)\right) \subset V_{M+N}.
\]
Therefore, it suffices to show $W\cap (V_{M+N})_1 =0$.

Recall that $M+N\cong A_2\otimes E_8$ has no roots. Thus,
\[
(V_{M+N})_1= span_\C\{h(-1)\cdot \mathbbm{1} \mid h\in (M+N)\otimes \C\}.
\]
However,
\[
\Omega_1 h(-1)\cdot \mathbbm{1} = \left(\omega_E +\frac{3}4 \omega_{M+N} -\frac{1}{12} \sum_{\al\in K(2)\atop 1\leq i,j\leq 3, i\neq j} e^{\eta_i(\al)-\eta_j(\al)}\right)_1 h(-1)\cdot \mathbbm{1}
= \frac{3}4 h(-1)\cdot \mathbbm{1} \neq 0
\]
for any $0\neq h\in (M+N)\otimes \C$.  Thus, $\left(\mathrm{Com}_{V_L}( L_{\widehat{sl}_9(\C)}(3,0)) \right)\cap  (V_{M+N})_1=0$ and we have $W_1=0$.
\end{proof}

\begin{rem}
Note that the lattice VOA $V_L$ also contains a subVOA isomorphic to $L_{\hat{E}_8 }(3,0)$, the level affine VOA associated to the Kac-Moody Lie algebra of type $E_8^{(1)}$. The central charge of $\Com_{V_L}(L_{\hat{E}_8 }(3,0))$ is $16/11$, which is the same as $U_{3C}$. In fact, it can be shown by the similar calculation as Lemma \ref{coma8} that $e_M$ and $e_N$ defined in Notation \ref{efg} are contained
in $\Com_{V_L}(L_{\hat{E}_8 }(3,0))$. Moreover, $$U_{3C}\cong VOA(e_M,e_N)= \Com_{V_L}(L_{\hat{E}_8 }(3,0)).$$
\end{rem}

\medskip

\subsection{A positive definite real form}
Next we shall show that the Ising vectors $e^{i,j}$, $0\leq i,j\leq 2,$ are contained in a positive definite real form
of $V_{E_8^3}$.

First we recall that the lattice VOA constructed in \cite{FLM} can be defined over $\R$.
Let $V_{L,\R}= S(\hat{\mathfrak{h}}^-_\R)\otimes \R\{L\}$ be the real lattice VOA associated to an even positive definite lattice, where $\mathfrak{h}=\R\otimes_\Z L$, $ \hat{\mathfrak{h}}^-= \oplus_{n\in \Z^+} \mathfrak{h}\otimes \R t^{-n}$. As usual, we use $x(-n)$ to denote $x\otimes t^{-n}$ for $x\in \mathfrak{h}$ and $n\in \Z^+$.

\begin{nota}\label{-1map}
Let $\theta: V_{L,\R} \to V_{L, \R}$ be defined by
\[
\theta( x(-n_1)\cdots x(-n_k)\otimes e^\al) = (-1)^k x(-n_1)\cdots x(-n_k)\otimes e^{-\al}.
\]
Then $\theta$ is an automorphism of $V_{L,\R}$, which is a lift of the $(-1)$-isometry of $L$ \cite{FLM}. We shall denote the $(\pm 1)$-eigenspaces of $\theta$ on $V_{L,\R}$ by $V_{L,\R}^\pm$.
\end{nota}

The following result is well-known \cite{FLM,M}.
\begin{prop}[cf. Proposition 2.7 of \cite{M}]\label{VLpd}
Let $L$ be an even positive definite lattice. Then
the real subspace
$\tilde{V}_{L,\R}=V_{L,\R}^+\oplus \sqrt{-1}V_{L,\R}^-$
is a real form of $V_L$.
Furthermore,
the invariant form on $\tilde{V}_{L,\R}$ is positive definite.
\end{prop}

Now apply the above theorem to the case $L=E_8^3$.  We have the following result.
\begin{prop}
Let $\tilde{V}_{E_8^3,\R}=V_{E_8^3 ,\R}^+\oplus \sqrt{-1}V_{E_8^3,\R}^-$. Then
$\tilde{V}_{E_8^3,\R}$ is a positive definite real form of $V_{E_8^3}$.
\end{prop}

The next lemma is clear by the definitions of $e_N, e_N$ and $e_{\tilde{N}}$.
\begin{lem}
The Ising vectors $e_M, e_N$ and $e_{\tilde{N}}$ defined in Notation \ref{efg} are contained
in $V_{E_8^3 ,\R}^+$.
\end{lem}

Recall the automorphism $\rho= \exp(\frac{2\pi i}3 ({\bf a}, -{\bf a},0)(0))$ defined in Definition \ref{rho}, where ${\bf a}$ is an element of $E_8$ such that
$K= \{\be\in E_8\mid \la \be, {\bf a}\ra \in 3\Z\} \cong A_8.$ Then we have the coset decomposition
\[
E_8= A_8 \cup (b+A_8) \cup (-b +A_8),
\]
where $b$ is a root of $E_8$ such that $\la b, {\bf a}\ra \equiv 1 \mod 3$.

Note that
\[
M=\{ (\al, -\al,0)\mid \al \in E_8\}\cong \sqrt{2}E_8
\]
and
\[
\tilde{K} =\{ (\al, -\al,0)\mid \al \in K\}\cong \sqrt{2}A_8.
\]

Set
\[
\begin{split}
X^0&: = \frac{1}3(e_M + \rho e_M +\rho^2 e_M), \\
X^1&:= \frac{1}3( e_M + \xi \rho e_M +\xi^2 \rho^2 e_M),\\
X^2&:= \frac{1}3(e_M + \xi^2 \rho e_M +\xi \rho^2 e_M) ,\\
\end{split}
\]
where $\xi=\exp{\frac{2 \pi i}3}= \frac{1}2( -1+\sqrt{-3})$.

The next lemma can be proved by the same calculations as in \cite{LYY}. Note that $\rho X^0=X^0$,
$\rho X^1=\xi^2 X^1$ and $\rho X^2= \xi X^2$.

\begin{lem}
The vector $X^0$ is contained in $V_{M,\R}^+$. Moreover,
\[
X^1=\frac{1}{32} \sum_{\gamma\in (b, -b,0)+\tilde{K}\atop \la \gamma, \gamma\ra =4} e^\gamma
\]
and
\[
X^2= \frac{1}{32} \sum_{\gamma\in -(b, -b,0)+\tilde{K}\atop \la \gamma, \gamma\ra =4} e^\gamma.
\]
Therefore, $X^1+X^2\in V_{M,\R}^+$ and $X^1- X^2\in V_{M,\R}^-$.
\end{lem}

\begin{lem}
The Ising vectors $e^{i,j}, 0\leq i,j\leq 2,$ are all contained in $\tilde{V}_{E_8^3,\R}$.
\end{lem}

\begin{proof}
By the discussion above, we have
\[
 \rho e_M = X^0 - \frac{1}2 (X^1+X^2) + \frac{1}2 \sqrt{-3} (X^1-X^2).
\]
Since $X^1+X^2\in V_{M,\R}^+$ and $X^1- X^2\in V_{M,\R}^-$, we have $\rho e_M \in \tilde{V}_{E_8^3,\R}$. Similarly, we also have $\rho^2 e_M, \rho e_N, \rho^2 e_N, \rho e_{\tilde{N}}, \rho^2 e_{\tilde{N}} \in \tilde{V}_{E_8^3,\R}$ as desired.
\end{proof}

\section{Parafermion VOA and $W$}

In this section, we shall show that the VOA $W$ defined in Notation \ref{Nota:W} is, in fact, equal to the commutant subVOA $\tilde{W}=\mathrm{Com}_{V_L}\left( L_{\widehat{sl}_9(\C)}(3,0)\right)$.
Recall from \cite{Lam2} that the lattice VOA $V_{A_3^8}$ contains a full subVOA $K(sl_3(\C),9)\otimes L_{\hat{sl}_9(\C)}(3,0)$, where
$K(sl_3(\C),9)$ is the parafermion VOA associated to the affine VOA $L_{\hat{sl}_3(\C)}(9,0)$.  Therefore, the VOA $\tilde{W}$ contains a full subVOA isomorphic to the parafermion VOA $K(sl_3(\C),9)$.

\subsection{Parafermion VOA}
First, we recall the definition of parafermion VOA from \cite{DW} (cf. \cite{DLY,DLWY}).

Let  $\mathfrak{g}$ be a finite dimensional simple Lie algebra
and $\widehat{\mathfrak{g}}$ the affine Kac-Moody Lie algebra
associated with $\mathfrak{g}$.  Let $\Pi=\{\al_1, \dots, \al_n\}$ be a set of simple roots and $\theta$ the highest root. Let $Q$ be the root lattice
of $\mathfrak{g}$.  For any positive integer $k$, we denote
\[
P_+^k(\mathfrak{g})= \{ \Lambda \in \Q\otimes_\Z Q\mid \la \al_i, \Lambda\ra \in \Z_{\geq 0} \text{ for all } i=1, \dots,n \text{ and } \la \theta, \Lambda\ra \leq k\},
\]
the set of dominant integral weights for $\mathfrak{g}$ with level $k$.

Let $L_{\widehat{\mathfrak{g}}}(k,\Lambda)$ be the
irreducible module of $\hat{\mathfrak{g}}$ with highest weight $\Lambda$ and level $k$.
Then $L_{\widehat{\mathfrak{g}}}(k,0)$ forms a simple VOA with the Virasoro element given by the Sugawara construction
\begin{equation}\label{suga}
\Omega_{\mathfrak{g},k}= \frac{1}{2 ( k +h^\vee)} \sum (u_i)_{-1}u^i,
\end{equation}
where $h^\vee$ is the dual Coxeter number, $\{u_i\}$ is a  basis of $\mathfrak{g}$ and
$ \{ u^i:= ( u_i) ^\ast   \}$ is the dual basis of $ \{ u_i \}$ with respect to the normalized Killing form (see \cite{FZ}). Moreover, the central charge of $L_{\widehat{\mathfrak{g}}}(k,0)$ is
\begin{equation}\label{cc}
\frac{k\dim\mathfrak{g}}{k+h^\vee}.
\end{equation}

 The vertex operator algebra
$L_{\widehat{\mathfrak{g}}}(k,0)$ contains a Heisenberg vertex
operator algebra corresponding to a Cartan subalgebra $\mathfrak{h}$
of $\mathfrak{g}$.
Let $M_{\hat{\mathfrak{h}}}(k, 0)$ be the vertex operator subalgebra of $L_{\widehat{\mathfrak{g}}}(k,0)$ generated by $h(-1)\cdot \mathbbm{1}$ for $h\in{\mathfrak{h}}$. The commutant $K(\mathfrak{g},k)$ of $M_{\hat{\mathfrak{h}}}(k, 0)$ in
$L_{\widehat{\mathfrak{g}}}(k,0)$ is called a parafermion vertex
operator algebra.

The VOA $L_{\widehat{\mathfrak{g}}}(k,0)$ is
completely reducible as an $M_{\hat{\mathfrak{h}}}(k, 0)$-module and we have a decomposition (see \cite{DW}).
\begin{lem}\label{Klambda}
For any $\lambda\in \mathfrak{h}^*$, let $M_{\hat{\mathfrak{h}}}(k, \lambda)$ be the irreducible highest weight module for $\hat{\mathfrak{h}}$ with a highest weight vector $v_\lambda$ such that $h(0)v_\lambda = \lambda(h)v_\lambda$ for $h\in \mathfrak{h}$. Denote
\[
K_{\mathfrak{g},k}(\lambda)= K_{\mathfrak{g},k}(0, \lambda) = \{v\in  L_{\widehat{\mathfrak{g}}}(k,0)\mid h(m)v =\lambda(h)\delta_{m,0} v \text{ for } h\in \mathfrak{h}, m\geq 0\}.
\]
Then we have
\[
L_{\widehat{\mathfrak{g}}}(k,0) =\bigoplus_{\lambda\in Q} K_{\mathfrak{g},k}(\lambda) \otimes M_{\hat{\mathfrak{h}}}(k, \lambda),
\]
where  $Q$ is the root lattice of $\mathfrak{g}$.
\end{lem}

Similarly, for any $\Lambda\in P_+^k(\mathfrak{g})$, we also have the decomposition.
\begin{lem}\label{KLl}
Denote $
K_{\mathfrak{g},k}(\Lambda, \lambda)=  \{v\in  L_{\widehat{\mathfrak{g}}}(k,\Lambda)\mid h(m)v =\lambda(h)\delta_{m,0} v \text{ for } h\in \mathfrak{h}, m\geq 0\}.
$ Then
\[
L_{\widehat{\mathfrak{g}}}(k,\Lambda) =\bigoplus_{\lambda\in \Lambda + Q} K_{\mathfrak{g},k}(\Lambda, \lambda) \otimes M_{\hat{\mathfrak{h}}}(k, \lambda).
\]
\end{lem}

\medskip

\subsection{A generating set}
In \cite{DW}, it is shown that the parafermion VOA $K(\mathfrak{g}, k)$ is generated by subVOAs isomorphic to $K(sl_2(\C), k)$. We first give a brief review of their work.

Let $\mathfrak{h}$ be a Cartan subalgebra of $\mathfrak{g}$ and let $\Delta_+$ be the set of all positive roots of $\mathfrak{g}$. Then
\[
\mathfrak{g}= \mathfrak{h}\oplus \bigoplus_{\al\in \Delta_+} (\C x_\al \oplus \C x_{-\al}),
\]
where $x_{\pm \al} \in \mathfrak{g}_{\pm\al} =\{ u\in\mathfrak{g}\mid [h,u]=\pm \al(h) u \text{ for all } h\in \mathfrak{h}\}$.

\begin{nota}\label{Pal}
For any $\al\in \Delta_+$, let  $h_\al=[x_\al, x_{-\al}]$.  Then $S_\al=span\{h_\al, x_\al, x_{-\al}\}$ is a Lie subalgebra of $\mathfrak{g}$ isomorphic to $sl_2(\C)$. Define
\[
\omega_\al  = \frac{1}{ 2k(k + 2)} (kh_\al(-2)1- h_\al(-1)2\mathbbm{1} + 2kx_\al(-1)x_{-\al}(-1)\mathbbm{1})
\]
and
\[
\begin{split}
W^3_\al = &k^2h_\al(-3)\mathbbm{1} + 3kh_\al(-2)h_\al(-1)\mathbbm{1} + 2h_\al(-1)^3\mathbbm{1}-
6kh_\al(-1)x_\al(-1)x_\al(-1)\mathbbm{1}\\
& +3k^2x_\al(-2)x_\al(-1)\mathbbm{1}-  3k^2x_\al(-1)x_\al(-2)\mathbbm{1}.
\end{split}
\]

We use  $P_\al$ to denote the vertex operator subalgebra of $K(\mathfrak{g}, k)$ generated by $\om_\al$ and $W^3_\al$ for $\al\in \Delta_+$.
\end{nota}

The next theorem is proved in \cite{DW}.

\begin{thm} [Theorem 4.2 of \cite{DW}]\label{dwtheorem}
The simple vertex operator algebra $K(\mathfrak{g}, k)$ is generated by $P_\al$, $\al\in \Delta_+$ and
$P_\al$ is a simple vertex operator algebra isomorphic to the parafermion vertex operator algebra $K(sl_2(\C), k)$ associated to $sl_2(\C)$.
\end{thm}

\subsection{Lattice VOA $V_{A_n^{k+1}}$} Next we recall an embedding of the VOA $$K(sl_{k+1}, n+1)\otimes L_{\hat{sl}_{n+1}(\C)}(k+1, 0)$$ into the lattice VOA $V_{A_n^{k+1}}$ from \cite{Lam2}.
\medskip

We use the standard model for the root lattice of type $A_\ell$. In particular,
\[
A_\ell=\{ \sum a_i\epsilon_i\in \Z^{\ell+1}\mid a_i\in \Z \text{ and } \sum_{i=1}^{\ell+1}a_i=0\},
\]
where $\epsilon_i$ is the row vector whose $i$-th entry is $1$ and all other entries are $0$. The dual lattice
\[
A_\ell^* =\bigcup_{i=0}^\ell \left ( \gamma_{A_\ell}(i)+A_\ell \right),
\]
where $ \gamma_{A_\ell}(i)= \frac{1}{\ell+1} \left( \sum_{j=1}^{\ell+1-i} i \epsilon_j - \sum_{j=\ell+1-i+1}^{\ell+1}
(\ell+1-i)\epsilon_j\right)$ for $i=0, \dots, \ell$.

\begin{nota}\label{nota3:1}
Let $n$ and $k$ be positive integers. We shall consider two injective maps   $\eta_i: \Z^{n+1}  \to \Z^{(n+1)(k+1)}$ and $\iota_i:
\Z^{k+1} \to \Z^{(n+1)(k+1)}$ defined by
\[
\eta_i(\epsilon_j) = \epsilon_{(n+1)(i-1) +j} \quad \text{ and } \quad  \iota_i(\epsilon_j)= \epsilon_{(n+1)(j-1) +i},
\]
for $i=1, \dots, k+1, j=1, \dots, n+1$.

Let $d_{k+1}= \sum_{j=1}^{k+1} \eta_j: \Z^{n+1}  \to \Z^{(n+1)(k+1)}$ and $\mu_{n+1}= \sum_{j=1}^{n+1}\iota_{j}: \Z^{k+1} \to
\Z^{(n+1)(k+1)}$. Then we have
\[
d_{k+1}(a_1, \dots, a_{n+1}) = (a_1, \dots, a_{n+1}, a_1, \dots, a_{n+1}, \dots, a_1, \dots, a_{n+1}),
\]
and
\[
\mu_{n+1}(a_1, \dots, a_{k+1}) = (a_1, \dots, a_1,a_2, \dots, a_2, \dots, a_{k+1},\dots, a_{k+1}).
\]
\end{nota}

Set $X=d_{k+1}(A_n)$ and $Y=\mu_{n+1}(A_k)$. Then $X\cong \sqrt{k+1}A_n$ and $Y\cong \sqrt{n+1}A_k$. Moreover, we have
\begin{equation}\label{ann}
\begin{split}
Ann_{A_{(n+1)(k+1)-1}}(Y) = \oplus_{i=1}^{k+1} \eta_i(A_n) \cong A_n^{k+1},\\
Ann_{A_{(n+1)(k+1)-1}}(X) = \oplus_{j=1}^{n+1}\iota_j(A_k) \cong A_k^{n+1},
\end{split}
\end{equation}
where $Ann_{A}(B)=\{ x\in A\mid \la x,y\ra =0 \text{ for all } y\in B\}$ is the annihilator of a sublattice $B$ in an integral lattice $A$.

\medskip

By the same construction as in Notation \ref{hef} (see also \cite[Chapter 13]{DL}), one can obtain subVOAs isomorphic to $L_{\hat{sl}_{n+1}(\C)}(k+1,0)$ and $L_{\hat{sl}_{k+1}(\C)}(n+1,0)$ in the lattice VOA $V_{A_{(k+1)(n+1)-1}}$.

The next proposition is well known in the literature \cite{KW,Lam2,NT}.
\begin{prop}\label{decL}
The VOA $L_{\hat{sl}_{n+1}(\C)}(k+1,0)$ and $L_{\hat{sl}_{k+1}(\C)}(n+1,0)$  are mutually commutative in the lattice VOA
$V_{A_{(k+1)(n+1)-1}}$. Moreover,
$L_{\hat{sl}_{n+1}(\C)}(k+1,0)\otimes L_{\hat{sl}_{k+1}(\C)}(n+1,0)$ is a full subVOA of  $V_{(n+1)(k+1)-1}$.
\end{prop}

\begin{rem}\label{Andec}
It is also known that the VOA $V_{\mu_{n+1}(A_k)}$ is contained in  the affine VOA $L_{\hat{sl}_{k+1}(\C)}(n+1,0)$ and $ K(sl_{k+1}(\C), n+1) = \Com_{L_{\hat{sl}_{k+1}(\C)}(n+1,0)} (V_{\mu_{n+1}(A_k)})$ (cf. \cite[Lemma 4.1]{Lam2}). Moreover, for any $\Lambda \in P^+_{n+1}(sl_{k+1}(\C))$, we have the decomposition
\begin{equation}\label{decan}
L_{\hat{sl}_{k+1}}(n+1,\Lambda) =\bigoplus_{\lambda\in  \frac{1}{n+1}(\mu_{n+1}(\Lambda+ A_k))/ {\mu_{n+1}(A_k)}}  K_{sl_{k+1}(\C),n+1} (\Lambda, (n+1)\bar{\lambda}) \otimes V_{\lambda+\mu_{n+1}(A_k)}
\end{equation}
as a module of $V_{\mu_{n+1}(A_k)}\otimes K(sl_{k+1}(\C),n+1)$, where $\bar{\lambda}\in \frac{1}{n+1}A_k$ such that $\mu_{n+1}(\bar{\lambda}) =\lambda$ (see \cite[Lemma 4.3]{Lam2}).

Note that it is shown in \cite[Theorem 14.20]{DL} that $K_{sl_{k+1}(\C),n+1} (\Lambda, (n+1)\bar{\lambda})$, for $\Lambda\in P_+^{n+1}(sl_{k+1}(\C))$, $\lambda\in (\mu_{n+1}(A_k))^*$, are irreducible $K(sl_{k+1}(\C),n+1)$-modules.
\end{rem}

\medskip

Next we consider the case when $n=8$ and $k=2$. Then $(n+1)(k+1)-1=26$. We shall study the decomposition of $\tilde{W} =\Com_{V_{E_8^3}}(L_{\hat{sl}_9(\C)}(3,0))$  as a $K(sl_3(\C),9)$-module.

Denote
$$\nu_1=\eta_1-\eta_2,\quad   \quad \nu_2=\eta_2-\eta_3, $$
and define $\mu=\mu_3:\Z^3 \to \Z^{27}$ by
\[
\mu (a_1, a_2, a_3) =(a_1, \dots, a_1, a_2, \dots, a_2,a_3, \dots a_3).
\]
Note that $Y=\mu (A_2)\cong 3A_2$ and $$Ann_{A_{26}}(Y)= \{ \al\in A_{26}\mid \la \al, \be\ra =0 \text{ for any }\be\in
Y\}\cong A_8^3.$$

Next we discuss the coset decomposition $Y+A_8^3$ in $A_{26}$.

\begin{lem}\label{a26}
Let  $\al_1=(1,-1,0)$ and $\al_2=(0,1,-1)$  be roots of $A_2$. Then we have
\[
\begin{split}
A_{26}  &= \bigcup_{0\leq i, j\leq 8}  \left( \left(-\frac{1}9 (i \mu (\al_1)+j\mu(\al_2))+Y\right) +
                  \left(\nu_1(\gamma_{A_8}(i))+\nu_2(\gamma_{A_8}(j))+A_8^3 \right) \right).
\end{split}
\]
\end{lem}

\begin{proof}
First we note that $ [A_{26}: Y+A_8^3]= \sqrt{( 9^2\cdot 3) \cdot 9^3/27}=9^2$.  Moreover,  we have
\[
-\frac{1}9 (i \mu (\al_1)+j\mu(\al_2)) +
                  \nu_1(\gamma_{A_8}(i))+\nu_2(\gamma_{A_8}(j)) = \sum_{k=10-i}^9 \iota_k(\alpha_1) +  \sum_{k'=10-i}^9 \iota_{k'}(\alpha_2).
\]
Note that $\sum_{k=10-i}^{9}  \iota_k (\alpha_p) \notin Y+A_8^3$ for any $i\neq 0, p=1,2$. Therefore,
\[
\left(  -\frac{1}9 (i \mu (\al_1)+j\mu(\al_2)) + Y\right ) +
                  \left(\nu_1(\gamma_{A_8}(i))+\nu_2(\gamma_{A_8}(j)) + A_8^3\right),
 \]
for $i,j =0, \dots, 8,$ give $9^2$ distinct cosets in $A_{26}/ (Y+A_8^3)$. Thus, we have the desired conclusion.
\end{proof}

\begin{lem}\label{e8overa8}
Let $\delta=\gamma_{A_8}(3)=\frac{1}3(1^6 -2^3) \in A_8^*$. Then for any $k, \ell =0, \pm 1$, we have \[
\begin{split}
\Com_{V_{(k\nu_1+\ell\nu_2)(\delta)+A_8^3}}(L_{\hat{sl}_9(\C)}(3,0))
&= \{ v\in V_{(k\nu_1+\ell\nu_2)(\delta)+A_8^3}\mid  \Omega_n v=0 \text{ for all } n\geq 0\}\\
&\cong K_{sl_3(\C), 9}(0, -3(k\al_1+\ell\al_2)).
\end{split}
\]
\end{lem}

\begin{proof}
By Lemma \ref{a26},
\[
V_{A_{26}} =\bigoplus_{0\leq i,j \leq 8}
V_{-\frac{1}9 (i \mu (\al_1)+j\mu(\al_2))+Y} \otimes V_{\nu_1(\gamma_{A_8}(i))+\nu_2(\gamma_{A_8}(j))+A_8^3}.
\]
Moreover, by \eqref{decan},
\[
L_{\hat{sl}_3}(9,0) =\Com_{V_{A_{26}}}(L_{\hat{sl}_9(\C)}(3,0))
=\bigoplus_{\lambda \in \frac{1}9 Y/Y} V_{\lambda+Y} \otimes K_{sl_3(\C), 9}(0, 9\bar{\lambda}).
\]
Take $i=3k$ and $j=3\ell$. Then we have
\[
\begin{split}
\Com_{V_{(k\nu_1+\ell\nu_2)(\delta)+A_8^3}}(L_{\hat{sl}_9(\C)}(3,0))
&\cong K_{sl_3(\C), 9}(0, 9\cdot -\frac{1}9 (3k\al_1 +3\ell \al_2))\\
 & = K_{sl_3(\C), 9}(0, -3(k\al_1+\ell\al_2))
 \end{split}
\]
as desired.
\end{proof}

\begin{lem}
We have the decomposition
\[
\tilde{W}= \Com_{V_{E_8^3}}(L_{\hat{sl}_9(\C)}(3,0)) = \bigoplus_{i,j=0, \pm 1} K_{sl_3(\C), 9}(0, 3(i\al_1+j\al_2)).
\]
\end{lem}

\begin{proof}
First we note that $M+N \cong A_2\otimes E_8$ and
\[
M+N= \bigcup_{0\leq k,\ell \leq 2} \left( (k\nu_1+\ell\nu_2)(\delta)+A_2\otimes A_8\right).
\]
Since $A_2\otimes A_8 \cong Ann_{A_8^3}(d_3(A_8))$ and $V_{d_3(A_8)}\subset L_{\hat{sl}_9(\C)}(3,0)$, we have
\[
\tilde{W}= \Com_{V_{E_8^3}}(L_{\hat{sl}_9(\C)}(3,0)) < V_{M+N}.
\]
The conclusion now follows from Lemma \ref{e8overa8}.
\end{proof}

\medskip
Now let $\al\in A_2$ be a root. Then $\Z\al\cong A_1$ and
\[
L(\al)=\oplus_{j=1}^9 \iota_j(\Z\al) \cong A_1^9 \subset A_{26}.
\]
Let $H_\al$ and $E_\al$ be defined as in Notation \ref{hef}. Then $\{H_\al, E_{\al}, -E_{-\al}\}$ forms a $sl_2$-triple in the lattice VOA $V_{A_1^9}< V_{A_{26}}$. Moreover, it generates a subVOA
$\mathcal{L}_\al$ isomorphic to the affine VOA $L_{\hat{sl}_2(\C)}(9,0)$. Let $M_\al(9,0)$ be the subVOA generated by $H_\al$. Then
\[
\mathcal{K}_\al:= \Com_{\mathcal{L}_\al}(M_\al(9,0)) \cong K(sl_2(\C), 9).
\]
Note also that $\mathcal{K}_\al= \Com_{\mathcal{L}_\al}(M_\al(9,0)) < \Com_{L_{\hat{sl}_3(\C)}(9,0)}( V_{\mu(A_2)})=K({sl}_3(\C),9).$

Set $h_\al=H_\al$, $x_\al=E_\al$ and $x_{-\al}= -E_{-\al}$. Then the elements $\om_\al$ and $W_\al^3$ defined in Notation \ref{Pal} are contained in $\mathcal{K}_\al$. In fact, $\mathcal{K}_\al$
is generated by $\om_\al$ and $W_\al^3$  (see \cite{DLY}).

\begin{thm}
The VOA $W$ defined in Notation \ref{eijGW} contains a full subVOA isomorphic to $K(sl_3(\C), 9)$.
\end{thm}

\begin{proof}
Recall that $W= VOA( e^{i,j}\mid  0\leq i,j\leq 2)$. We also have
\[
M=(\eta_1-\eta_2)(E_8) ,\quad  N= (\eta_2-\eta_3)(E_8),\quad    \tilde{N}=(\eta_1-\eta_3)(E_8).
\]

Let $\al_1=(1,-1,0)$, $\al_2=(0,1,-1)$ and $\al_3= \al_1+\al_2=(1,0,-1)$ be the positive roots of $A_2$. Then by the same calculations as in \cite{LYY}, it is straightforward to   verify that
\[
\mathcal{K}_{\al_1} < VOA(e_M, \rho e_M),\quad  \mathcal{K}_{\al_2} < VOA(e_N, \rho e_N),\quad  \mathcal{K}_{\al_3} < VOA(e_{\tilde{N}}, \rho e_{\tilde{N}}),
\]
where $e_M, e_N, e_{\tilde{N}}$ and $\rho$ are defined as in Notation \ref{efg}.

By Theorem \ref{dwtheorem}, $\mathcal{K}_{\al_1}, \mathcal{K}_{\al_2}$ and $\mathcal{K}_{\al_3}$ generates a subVOA isomorphic to $K(sl_3(\C), 9)$ in $W$. It is a full subVOA of $W$ because they have the same central charge.
\end{proof}

\begin{thm}
 We have $W=\tilde{W}= \Com_{V_{E_8^3}}(L_{\hat{sl}_9(\C)}(3,0))$.
\end{thm}

\begin{proof}
By the previous lemma, the subVOA $W$ contains
$K(sl_3(\C),9)$ as a full subVOA.

Therefore, it suffices to show that $K_{sl_3(\C), 9}(0, 3(i\al_1+j\al_2))$ is contained in $W$ for
any $i,j =0, \pm 1$.

By \cite[Proposition 2.2]{LYY},
\[
X_{\nu_1}^+ =\frac{1}{32}\sum_{\gamma\in \nu_1(\delta) + \nu_1(A_8)\atop \la \gamma, \gamma \ra=4} e^\gamma  \qquad \text{ and } \qquad X_{\nu_1}^- =\frac{1}{32}\sum_{\gamma\in -\nu_1(\delta)+\nu_1(A_8)\atop \la \gamma, \gamma \ra=4} e^\gamma
\]
are contained in $VOA(e_M, \rho e_M)< W$.
Moreover, it is straightforward to verify that
\[
X_{\nu_1}^+ \in \Com_{V_{\nu_1(\delta)+A_8^3}}(L_{\hat{sl}_9(\C)}(3,0)) \cong   K_{sl_3(\C), 9}(0, -3\al_1)
\]
and
\[
X_{\nu_1}^- \in \Com_{V_{-\nu_1(\delta)+A_8^3}}(L_{\hat{sl}_9(\C)}(3,0)) \cong  K_{sl_3(\C), 9}(0, 3\al_1).
\]
Therefore, $W$ contains $K_{sl_3(\C), 9}(0, \pm 3\al_1)$ as  $K(sl_3(\C),9)$-submodules.
Similarly, $W$ also contains $K_{sl_3(\C), 9}(0, \pm 3\al_2)$ and $ K_{sl_3(\C), 9}(0, \pm 3(\al_1+\al_2))$
as $K(sl_3(\C),9)$-submodules.

Moreover, it is clear that $0 \neq (X_{\nu_1}^+)_{-3} (X_{\nu_2}^-) \in  V_{(\nu_1-\nu_2)(\delta)+A_8^3} $.  Since $X_{\nu_1}^+$ and $X_{\nu_2}^-$ are contained in the commutant of $L_{\hat{sl}_9(\C)}(3,0)$,  we have $(X_{\nu_1}^+)_{-3} (X_{\nu_2}^-) \in
\Com_{V_{(\nu_1-\nu_2)(\delta)+A_8^3}}(L_{\hat{sl}_9(\C)}(3,0))$.
Hence $W$ contains $K_{sl_3(\C), 9}(0, 3(\al_1-\al_2))$. Similarly, $K_{sl_3(\C), 9}(0, 3(\al_2-\al_1))$ is contained in $W$, also.
\end{proof}

\end{document}